\documentclass[11pt,leqno]{article}
\usepackage{amsthm,amsfonts,amssymb,amsmath,oldgerm}
\usepackage{epsfig}
\numberwithin{equation}{section}
\usepackage{epstopdf}

\usepackage{appendix}
\renewcommand\appendix{\par
\setcounter{section}{0}%
\setcounter{subsection}{0}%
\setcounter{table}{0}
\setcounter{figure}{0}
\setcounter{equation}{0}
\gdef\thetable{\Alph{table}}
\gdef\thefigure{\Alph{figure}}
\gdef\theequation{\Alph{section}-\arabic{equation}}
\section*{Appendix}
\gdef\thesection{\Alph{section}}
\setcounter{section}{0}} 

\renewcommand\a{\alpha}

\def\eps{\varepsilon }

\renewcommand\a{\alpha}

\def\eps{\varepsilon}

\newcommand\br{\begin{remark}}
\newcommand\er{\end{remark}}
\newcommand\bp{\begin{pmatrix}}
\newcommand\ep{\end{pmatrix}}
\newcommand{\be}{\begin{equation}}
\newcommand{\ee}{\end{equation}}
\newcommand\ba{\begin{equation}\begin{aligned}}
\newcommand\ea{\end{aligned}\end{equation}}

\newcommand{\bap}{\begin{app}}
\newcommand{\eap}{\end{app}}
\newcommand{\begs}{\begin{exams}}
\newcommand{\eegs}{\end{exams}}
\newcommand{\beg}{\begin{example}}
\newcommand{\eeg}{\end{exaplem}}
\newcommand{\bpr}{\begin{proposition}}
\newcommand{\epr}{\end{proposition}}
\newcommand{\bt}{\begin{theorem}}
\newcommand{\et}{\end{theorem}}
\newcommand{\bc}{\begin{corollary}}
\newcommand{\ec}{\end{corollary}}
\newcommand{\bl}{\begin{lemma}}

\newcommand{\bd}{\begin{definition}}
\newcommand{\ed}{\end{definition}}
\newcommand{\brs}{\begin{remarks}}
\newcommand{\ers}{\end{remarks}}

\newcommand{\sgn}{\text{\rm sgn}}
\newtheorem{theorem}{Theorem}[section]
\newtheorem{proposition}[theorem]{Proposition}
\newtheorem{corollary}[theorem]{Corollary}
\newtheorem{lemma}[theorem]{Lemma}

\theoremstyle{remark}
\newtheorem{remark}[theorem]{Remark}
\theoremstyle{definition}
\newtheorem{definition}[theorem]{Definition}

\newtheorem{example}[theorem]{Example}

\newcommand{\RM}{\mathbb{R}}
\newcommand{\ZM}{\mathbb{Z}}

\newcommand{\CM}{\mathbb{C}}

\newcommand{\beq}{\begin{equation}}
\newcommand{\eeq}{\end{equation}}

\title{
Stability of Small Periodic Waves in Fractional KdV Type Equations 
}

\author{\sc
Mathew A. Johnson\thanks{Department of Mathematics, University of Kansas, 1460 Jayhawk Boulevard, 
Lawrence, KS 66045; matjohn@math.ku.edu.
}
}
\begin{document}
\maketitle


\begin{center}
\end{center}

\begin{center}
{\bf 2010 MR Subject Classification}: 35Q53, 35B35, 35B10, 35P99.
\end{center}


\begin{abstract}
We consider the effects of varying dispersion and nonlinearity on the stability of periodic traveling
wave solutions of nonlinear PDE of KdV-type, including generalized KdV and Benjamin-Ono equations.  
In this investigation, we consider the spectral
stability of such solutions that arise as small perturbations
of an equilibrium state.  A key feature of our analysis is the development of a nonlocal Floquet-like theory
that is suitable to analyze the $L^2(\RM)$ spectrum of the associated linearized operators.  Using spectral perturbation theory then,
we derive a relationship between the power of the nonlinearity and the symbol of the 
fractional dispersive operator that determines the spectral stability and instability to arbitrary small localized perturbations.
\end{abstract}



\pagestyle{myheadings}
\thispagestyle{plain}
\markboth{Mathew A. Johnson}{Stability of Small Periodic Waves of KdV Type}

\section{Introduction}
In this paper, we are concerned with the spectral stability and instability of periodic traveling wave solutions 
$u(x,t)=u(x-ct)$ of a class of scalar evolution equations of the form
\begin{equation}\label{e:kdv}
u_{t}+\left(-\Lambda^\alpha u+u^{p+1}\right)_{x}=0,\quad x,t\in\RM,
\end{equation}
where subscripts denote partial differentiation, $u=u(x,t)$ is a real-valued function,
and where the pseudodifferential operator $\Lambda=\sqrt{-\partial_x^2}$, referred to as Calderon's operator,
is of order one and is defined by its Fourier multiplier as $\widehat{\Lambda u}(k)=|k|\hat{u}(k)$;
throughout our analysis, the circumflex denotes the Fourier transform taken either on $\RM$
or an appropriate one-dimensional torus, depending on the context.  By inspection of the 
Fourier symbol, we see that Calderon's operator can alternatively be defined
as $\Lambda=\mathcal{H}\partial_x$ where $\mathcal{H}$ denotes the Hilbert transform being
applied either on the line or the torus, depending on the context, and in the $x$ variable.
Here, we consider $\alpha>\frac{1}{2}$ and either $p\in\mathbb{N}$ or $p=\frac{m}{n}$ with
$m$ and $n$ being even and odd natural numbers, respectively.

Equations of the form \eqref{e:kdv} arise naturally in the modeling
of unidirectional propagation of weakly nonlinear dispersive waves of long-wavelength, in which
case $u$ represents the wave profile or its velocity and the variables $x$ and $t$ are proportional
to the distance in the direction of propagation and the elapsed time, respectively.
In this context, the parameter $\alpha>0$ characterizes
the linear dispersion about the zero state; in particular, letting $c(\xi)$ denote the phase velocity
of plane waves with frequency $\xi$, we find the dispersion relation $c(\xi)=|\xi|^\alpha$
for equations of the form \eqref{e:kdv}.  It is the goal
of this paper to derive conditions on the dispersion parameter $\alpha$ and the power $p$ of the nonlinearity
for which the small amplitude periodic traveling wave solutions of \eqref{e:kdv} are stable.

Arguably the most common and well studied example of an equation of form \eqref{e:kdv} is the generalized KdV equation
\begin{equation}\label{e:gkdv}
u_t+u_{xxx}+\left(u^{p+1}\right)_x=0
\end{equation}
corresponding to $\alpha=2$.  When $p=1$, \eqref{e:gkdv} was proposed by Korteweg and de Vries \cite{KdV}
in 1895 to model the unidirectional propagation of surface water waves of small amplitude and long wavelengths in a channel.
When $p=2$, \eqref{e:gkdv} corresponds to the modified KdV equation, which arises as a model for large amplitude
internal waves in a density stratified medium, as well as for Fermi-Pasta-Ulam lattices with bistable
nonlinearity.  In both of these cases, the PDE is completely integrable and hence the Cauchy problem can, in principle,
be completely solved via the inverse scattering transform.

Another important class of equations of form \eqref{e:kdv} arises when $\alpha=1$, in which
case we recover the generalized Benjamin-Ono equation
\begin{equation}\label{e:gbo}
u_t+\left(-\mathcal{H}u_x+u^{p+1}\right)_x=0.
\end{equation}
When $p=1$, Benjamin \cite{Ben} and Ono \cite{Ono} independently derived \eqref{e:gbo}
as a model for the unidirectional propagation of internal waves in deep water.
In this case, \eqref{e:gbo} is also completely integrable.  
Further examples with varying values of $\alpha$ can be derived in the context of shallow water theory
by assuming different order relationships between the small quantities $\gamma:=\frac{a}{h}$ and $\beta:=\frac{h^2}{\lambda^2}$, where
$h$ denotes the depth of the water at rest, and $a$ and $\lambda$ denote characteristic amplitudes and wavelengths
of the waves searched for, respectively, corresponding to modeling in the small amplitude and small wavelength
regime.  
Of particular interest, although outside the scope of our analysis, we point out that in the case\footnote{Notice
that the operator $\partial_x\Lambda^\alpha$ is non-singular for all $\alpha\geq-1$.} $\alpha=-\frac{1}{2}$, 
\eqref{e:kdv} was recently shown by Hur \cite{H} to approximate up to quadratic order the surface water wave problem in two
spatial dimensions in the infinite depth case, hence generalizing Whitham's equation in \cite{W}.

The stability of the solitary wave solutions of equations of the form \eqref{e:kdv} have a long and rich history,
dating back to the novel work of Benjamin in \cite{Ben2} in which the stability of such waves
in the case $\alpha=2$ was established for $p<4$.  The stability analysis for $\alpha\geq 1$ was carried
out by Bona, Souganidis, and Strauss in \cite{BSS} where the authors extended the seminal work of 
Grillakis, Shatah, and Strauss \cite{GSS} to equations of form \eqref{e:kdv} where, most notably, the symplectic
form in the Hamiltonian structure fails to be invertible.  In \cite{BSS}, the authors establish the nonlinear orbital stability
of the solitary wave solutions of \eqref{e:kdv} when $\alpha\geq 1$ for $p<2\alpha$ and the instability of such solutions
for $p>2\alpha$.

In contrast to its solitary wave counterpart, the stability theory for $T$-periodic traveling 
wave solutions of equations of the form \eqref{e:kdv} has received
considerably less attention, even in the classical case $\alpha=2$, 
and this theory is still far from complete.  Within this context, stability results typically fall
into one of two categories: spectral stability to perturbations in $C_b(\RM)$ or $L^2(\RM)$ (see \cite{BrJ,BrJZ,DB,DN1,HK}), and nonlinear
orbital stability to perturbations in $L^2_{\rm per}([0,nT])$ for some $n\in\mathbb{N}$ 
(see \cite{A,ABS,AN,BrJK,DN1,DN2,DK,J}).  The majority of the nonlinear stability results restrict to the co-periodic case, i.e. perturbations
in $L^2_{\rm per}([0,nT])$ when $n=1$, a clearly very restrictive class of perturbations,
in which case authors are
often
able to use adaptations of the stability theory in \cite{GSS} to establish orbital stability.  
The only examples the author is aware of that establishes orbital stability for $n>1$ 
are due to Deconinck and Nivala \cite{DN1,DN2} where the authors
consider the KdV and mKdV equations, relying heavily in both cases on the complete integrability of the
governing equations. 

Concerning the spectral stability of such solutions, Deconinck and Bottman
\cite{DB}
established the spectral stability in $L^2(\RM)$ 
of periodic traveling wave solutions of the classic KdV equation ($\alpha=2$ and $p=1$),
while Deconinck and Nivala 
\cite{DN1,DN2}
established such a result for the modified KdV equation ($\alpha=2$ and $p=2$).
For more general power law nonlinearities, when $\alpha=2$ Bronski and Johnson \cite{BrJ}analyzed the spectral
stability of such periodic traveling waves of \eqref{e:gkdv} and derived there a geometric index for the stability
and instability of such waves to perturbations in $L^2_{\rm per}([0,nT])$ with $n\gg 1$, while Haragus and Kapitula \cite{HK}
established the spectral stability of such waves of sufficiently small amplitude when $p<2$ and spectral
instability when $p>2$; as mentioned previously, it is this class of small amplitude periodic waves that we will be concerned with
here.

It is important to note that all of the periodic stability analyses described above are valid only in the local
case $\alpha=2$.  Such results for $\alpha\in(0,2)$ seem to be very few.  Most notably, in \cite{AN}
Angulo and Natali investigate the nonlinear stability of periodic traveling wave solutions to equations of the form
\eqref{e:kdv} when subject to perturbations with the same period as the underlying wave\footnote{In fact, they 
consider more general nonlocal dispersive operators than what are considered here.}.  In particular, they
establish the nonlinear stability of periodic traveling waves of the Benjamin-Ono equation, corresponding
to $\alpha=1$ and $p=1$ in \eqref{e:kdv} to perturbations with the same period as the underlying wave. 
More recently, Hur and Johnson in \cite{HJ} verified for $\alpha\in(\frac{1}{3},2]$ and $p=1$ the nonlinear stability of $T$-periodic traveling wave solutions
of \eqref{e:kdv} to $T$-periodic perturbations when the underlying wave arises as a constrained energy minimizer.
As far as the author is aware, these analyses are the only rigorous results
concerning the nonlinear stability of periodic waves in equations of the form \eqref{e:kdv} 
when 
the dispersive operator is nonlocal.
Furthermore, the author is not aware of any rigorous results concerning the spectral stability
of periodic waves in equations of the form \eqref{e:kdv} when $\alpha\neq 2$.  

It is the intent of the  current paper to investigate the spectral stability of periodic
traveling wave solutions of \eqref{e:kdv} of sufficiently small amplitude when subject to
arbitrarily small localized, i.e. integrable, perturbations.  Our main result, stated in Theorem \ref{t:kdvmain} below,
states that such small amplitude periodic traveling wave trains are spectrally
stable if $\alpha>1$ and $1\leq p<p^*(\alpha)$, where the function $p^*(\alpha)$ is defined
explicitly in \eqref{e:critpower} below, and is spectrally unstable if either $\alpha\in(\frac{1}{2},1)$
or $\alpha>1$ and $p>p^*(\alpha)$.  A plot of $p^*(\alpha)$ is givein in Figure \ref{f:critp}, and from
this it is evident that for $p$ sufficiently large all such small periodic traveling wave solutions 
are necessarily spectrally unstable.  Furthermore, for $p\in\left(1,P_{\rm max}\right)$, where $P_{\rm max}:=\max_{\alpha\geq 1}p^*(\alpha)\approx 2.19$,
there exist numbers $\alpha_-(p)<\alpha_+(p)$ such that such small periodic traveling waves
are spectrally stable provided $\alpha\in(\alpha_-(p),\alpha_+(p))$ and are spectrally unstable
if $\alpha\notin[\alpha_-(p),\alpha_+(p)]$.  The fact that there is an upper bound the on the admissible dispersion parameters $\alpha$
corresponding to stability for a given $p\in(1,P_{\rm max})$ is a striking new feature in the periodic stability analysis
of models of the form \eqref{e:kdv}, and it stands in direct contrast with the solitary wave theory.  It is likely that
this is due to the fundamental difference between the nature of disperison on the line and ``dispersion" on the circle: here, we do not
attempt to give an explanation for this difference, and leave this instead as an interesting open question.

\begin{figure}[htbp]
\centering
\includegraphics{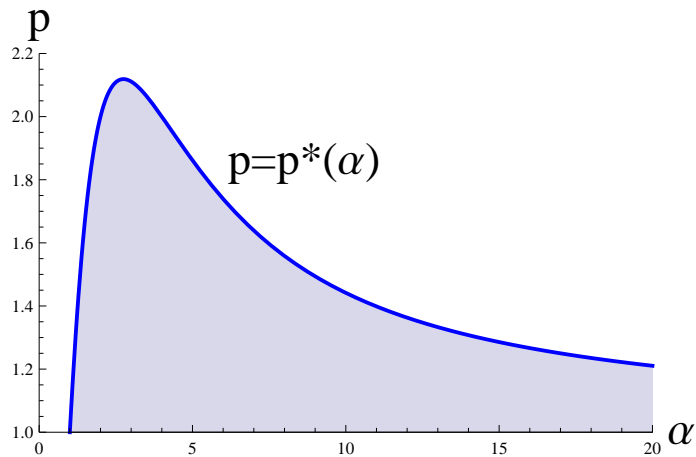}
\caption{A plot of the critical nonlinearity $p^*(\alpha)$ as a function of the dispersion parameter $\alpha$,
plotted for $\alpha\geq 1$.  Notice that $p^*(\alpha)$ decreases monotonically to $1$ as $\alpha\to\infty$.  
The shaded area corresponds to the region of stability for the small amplitude
periodic traveling waves considered here.  For $\alpha\in(1/2,1)$, we see $p^*(\alpha)<1$ and hence all such small periodic
traveling waves are spectrally unstable.}
\label{f:critp}
\end{figure}

Concerning specific models of the form \eqref{e:kdv}, in the classical case $\alpha=2$, Theorem \ref{t:kdvmain} 
recovers the result of Haragus and Kapitula \cite{HK} that such waves are spectrally stable if $1\leq p<2$ and are
spectrally unstable if $p>2$.  Furthermore, since $p^*(1)=1$, it follows that our analysis is not sufficient to conclude
spectral stability or instability in the classical Benjamin-Ono equation, corresponding to $\alpha=p=1$ in
\eqref{e:kdv}: see Remark \ref{r:bo} for more details.  Finally, fixing the nonlinearity $p$ instead, we see
that when $p=1$ all such small periodic traveling waves are spectrally stable for all $\alpha>1$, while when
$p=2$ such waves are spectrally stable if $\alpha\in(2,4)$ and are spectrally unstable if $\alpha\notin[2,4]$.
It is important to note that our analysis demonstrates that \emph{all} such small $T$-periodic traveling
wave solutions of \eqref{e:kdv} are spectrally stable to small $T$-periodic perturbations,
so that the spectral instability detected in Theorem \ref{t:kdvmain}
is necessarily of sideband type, occurring in $L^2_{\rm per}([0,nT])$ for $n> 1$, 
and is hence impossible to detect by co-periodic stability analyses.

While our analysis is similar to that given in \cite{HK} in the local case $\alpha=2$, relying
on perturbation arguments of the spectrum from the easily analyzed constant state, it is complicated by the facts
that (1) the existence theory no longer follows by elementary phase plane analysis, and (2) the absence of a 
suitable Floquet theory for nonlocal differential equations with periodic coefficients, which is necessary
to analyze the essential spectrum of the linearized operators obtained from linearizing \eqref{e:kdv} about
a given periodic traveling wave.  In the forthcoming analysis, we resolve (1) by using a Lyapunov-Schmidt reduction
argument similar to that given in \cite{HLS} in the context of the fifth-order Kawahara equation.
For (2), we utilize the inverse Bloch-Fourier representation of functions in $L^2(\RM)$, which
is well known in the analysis of Schr\"odinger operators with periodic potentials \cite{RS} and has been
extensively used in the stability analysis of periodic wave trains in dissipative systems (see \cite{BJNRZ,JNRZ1,JNRZ2,JNRZ3,G,M} 
and references therein), 
to show that, even in this nonlocal setting, the essential spectrum of the linearized operators acting on $L^2(\RM)$ can
be continuously parameterized by the eigenvalues of a one-parameter family of Bloch operators acting
on a periodic domain.  This spectral characterization extends that introduced by Gardner \cite{G}
in the local setting, and is valid for considerably more general nonlocal operators than
what is considered here.

The outline of this paper is as follows.  In Section \ref{s:existence}, we prove the existence and determine asymptotic
expansions of periodic traveling wave solutions of \eqref{e:kdv} of sufficiently small amplitude.  The existence
argument is based on an appropriate Lyapunov-Schmidt reduction argument, the details of which are included in Appendix \ref{a:exist}.
Section \ref{s:stability} contains our main stability results, beginning with a careful characterization of the essential
spectrum of the linearized operators acting on $L^2(\RM)$ in terms of the eigenvalues of a one-parameter family
of Bloch operators acting on periodic functions.  We then use spectral perturbation arguments to analyze
the spectrum of the small amplitude periodic wave by considering the associated Bloch operators as small
perturbations of those with constant coefficients obtained from linearizing about the nearby constant state.  As such,
the restriction to periodic waves of sufficiently small amplitude is essential in our argument.  In particular, our analysis
gives no information about the stability or instability of periodic waves with large amplitude when subject
to small localized perturbations.  Finally, we conclude with an appendix in which we give the proofs of the
existence result and the Bloch-wave decomposition.

\section{Existence of Periodic Traveling Waves}\label{s:existence}

In this section we analyze the set of periodic traveling wave solutions of \eqref{e:kdv} of the 
form
\[
u(x,t)=u(x-ct),\quad c\in\RM,
\]
where the function $u(\cdot)$ is a real-valued periodic function of its argument.
Due to the scaling properties of \eqref{e:kdv} we can with out loss of generality assume that $c=1$,
in which case such solutions of \eqref{e:kdv} arise as stationary solutions of the PDE
\begin{equation}\label{e:travkdv}
u_{t}+\left(-\Lambda^\alpha u-u+u^{p+1}\right)_{x}=0,
\end{equation}
or, equivalently, after performing a single integration, solutions of the nonlocal profile equation
\be\label{e:travode}
-\Lambda^\alpha u-u+u^{p+1}=b
\ee
where $b\in\RM$ is a constant of integration taken to be in general non-zero.  

We begin by considering the equilibrium solutions of \eqref{e:travode} for $|b|\ll 1$.  Notice when $b=0$, there are
in general two non-negative equilibrium solutions $u=0$ and $u=1$.  In the classical case
when $\alpha=2$, it follows by straight forward phase plane analysis that in the $(u,u')$ phase
plane, $u=0$ is a saddle point associated with a homoclinic orbit (solitary wave), while the equilibrium
$u=1$ is a nonlinear center.  Thus, when $\alpha=2$ and $b=0$ it is clear that there exists a one-parameter family of periodic orbits
$\left\{(u_\gamma,u'_\gamma)\right\}_{\gamma\in[0,1)}$ of period $T_\gamma$ such that
\[
\lim_{\gamma\to 0^+}T_\gamma=\frac{2\pi}{\sqrt{p}},\quad\lim_{\gamma\to 1^-}T_\gamma=+\infty.
\]
Moreover, these qualitative features persist for $|b|\ll 1$.

When $\alpha\neq 2$, however, this classical picture breaks down as we can no longer make
sense of the phase space in the same way.  Nevertheless, for $|b|\ll 1$ there exists a unique equilibrium
solution $u=Q_b$ of \eqref{e:travode} continuing from $Q_0=1$; indeed, a straightforward calculation
provides the expansion
\[
Q_b=1+\frac{1}{p}b-\frac{p+1}{2p^2}b^2+\mathcal{O}(|b|^3)
\]
valid for $|b|\ll 1$.  Seeking nearby periodic solutions of \eqref{e:travode}, we set $u(x)=P(kx)$, where $P$ 
is $2\pi$-periodic and $k>0$ denotes the wave number, and require that $P$ and $k$ satisfy
the rescaled nonlocal profile equation
\[
-k^\alpha\Lambda^\alpha P-P+P^{p+1}=b
\]
posed on a $2\pi$-periodic domain.  Here, we consider for each $s>0$ the operator $\Lambda^s$ as being defined on the torus.  
In particular,
we consider $\Lambda^s$ as a closed operator on $L^2(\RM/2\pi\ZM)$ with
dense domain $H^s(\RM/2\pi\ZM)$ being defined via Fourier series as 
\begin{equation}\label{e:fourier}
\Lambda^sf(x)=\sum_{k\in\ZM\setminus\{0\}}|k|^se^{ikx}\hat{f}(k),~~s\geq 0.
\end{equation}
It is clear from this definition that $\Lambda^s$ is invertible for each $s>0$ when restricted
to the mean-zero subspace of $H^s(\RM/2\pi\ZM)$, and we define its inverse $\Lambda^{-s}$ on this subspace
via Fourier series as
\[
\Lambda^{-s}f(x)=\sum_{k\in\ZM\setminus\{0\}}\frac{e^{ikx}\hat{f}(k)}{|k|^s},~~s>0;
\]
see \cite{RSt} for a recent interesting discussion of $\Lambda^s$ acting on the torus.
With these definitions, for $\alpha>\frac{1}{2}$ one can use a Lyapunov-Schmidt reduction to verify the existence
of a two-parameter family of small amplitude even periodic traveling wave solutions $u_{a,b}$ of \eqref{e:travode}
existing in a neighborhood of the equilibrium solution $u=Q_b$; see Appendix \ref{a:exist} for details.

In the parameterization of the periodic solutions $u_{a,b}(x)=P_{a,b}(k_{a,b}x)$ of \eqref{e:travode} given by Theorem \ref{t:existence},
the functions $P_{a,b}$ are $2\pi$-periodic even solutions of the rescaled nonlocal profile equation
\begin{equation}\label{e:profileeqn}
-k_{a,b}^{\alpha}\Lambda^\alpha v-v+v^{p+1}=b
\end{equation}
such that $P_{0,b}=Q_b$ and $k_{0,b}^\alpha=(p+1)Q_b^p-1$; here, we have used
the fact that $k_{a,b}>0$ for $|(a,b)|\ll 1$.  Furthermore,
the parameter $a$ is precisely the first Fourier coefficient $P_{a,b}$, 
and we have the relations $P_{a,b}(z+\pi)=P_{-a,b}(z)$ and $k_{a,b}=k_{-a,b}$ valid for all $|(a,b)|\ll 1$.
Taking into account the translation invariance of \eqref{e:travode}, it follows that for a fixed $\alpha>\frac{1}{2}$ 
we can find a three-parameter family of small amplitude periodic traveling wave solutions of 
\eqref{e:travkdv}.

A key feature of the functions $P_{a,b}$ and $k_{a,b}$ described in Theorem \ref{t:existence} is that these depend
analytically on $a$ and $b$ for $|(a,b)|\ll 1$.  The next lemma exploits this fact
to provide us with explicit expansions of the functions $P_{a,b}$ and $k_{a,b}$ valid for sufficiently small
$a$ and $b$.  These expansions are crucial in the forthcoming stability analysis.

\begin{lemma}\label{l:smallampexpand}
For sufficiently small $a,b\in\RM$ and $\alpha>\frac{1}{2}$, the functions $P_{a,b}$ and $k_{a,b}$
in Theorem \ref{t:existence} can be expanded as
\begin{align*}
P_{a,b}(x)&=Q_b+\cos(z)a+\frac{p+1}{4}\left(\frac{1}{2^\alpha-1}\cos(2z)-1\right)a^2+\mathcal{O}(|a|(a^2+b^2))\\
k_{a,b}^\alpha &=k_{0,b}^\alpha-\frac{p(p+1)(2^\alpha(p+3)-2(p+2))}{8(2^\alpha-1)}a^2+\mathcal{O}(|a|^3+|b|^3),
\end{align*}
where $k_{0,b}^\alpha=(p+1)Q_b^p-1$.
\end{lemma}

\begin{proof}
Taking $b=0$ for now and recalling Theorem \ref{t:existence}, 
we begin with the following small amplitude ansatz for $P_{a,0}$ and $k_{a,0}$
\begin{align*}
P_{a,0}(z)&=1+a\cos(z)+a^2v_2(z)+a^3v_3(z)+\mathcal{O}(a^4)\\
k_{a,0}^\alpha&=p+a^2k_1+\mathcal{O}(a^4)
\end{align*}
where each $v_j$ is even and $2\pi$-periodic in the $z$ variable.  Substituting these expansions
into the profile equation \eqref{e:profileeqn} yields a hierarchy of compatibility conditions.
The $\mathcal{O}(1)$ equation is trivially satisfied, while the $\mathcal{O}(|a|)$ equation reads
\[
\left(-\Lambda^\alpha +1\right)\cos(z)=0
\]
which again holds.
The $\mathcal{O}(a^2)$ equation now reads
\[
-\Lambda^\alpha v_2+v_2=-\frac{p(p+1)}{4}\left(1+\cos(2z)\right)
\]
which, using \eqref{e:fourier}, is seen to have even solutions of the form
\[
v_2(z)=\frac{p(p+1)}{4}\left(\frac{1}{2^\alpha-1}\cos(2z)-1\right)+A\cos(z)
\]
for any $A\in\RM$.
Notice, however, that we must take $A=0$ by the definition of the parameter $a$.
Continuing to the $\mathcal{O}(a^3)$ equation, we find
\begin{align*}
p\left(-\Lambda^\alpha+1\right)v_3&=k_1\cos(z)-\frac{p(p+1)^2}{8(2^\alpha-1)}\left(\cos(z)+\cos(3z)\right)+\frac{p(p+1)^2}{4}\cos(z)\\
&\quad-\frac{p(p-1)(p+1)}{24}\left(3\cos(z)+\cos(3z)\right),
\end{align*}
which is readily seen to have a resonant solution containing a term proportional to $z\sin(z)$ unless 
\[
k_1=\frac{p(p+1)^2}{8(2^\alpha-1)}-\frac{p(p+1)^2}{4}+\frac{p(p-1)(p+1)}{8},
\]
i.e. unless $k_1$ is chosen so that
\[
k_1=-\frac{p(p+1)\left(2^\alpha(p+3)-2(p+2)\right)}{8(2^\alpha-1)}.
\]
This completes the expansions to the desired order when $b=0$.  The expansions
when $|b|\ll 1$ are obtained similarly.
\end{proof}

\section{Spectral Stability to Localized Perturbations}\label{s:stability}

We are now ready to begin discussing the stability
analysis of the small amplitude periodic solutions $P_{a,b}(k_{a,b}\cdot)$ constructed in the previous section under the flow induced
by the PDE \eqref{e:travkdv}.  
As mentioned in the introduction, we are interested here in the spectral
stability of such waves when subject to small localized, i.e. integrable, perturbations on $\RM$.  In the classical
case when $\alpha\in 2\mathbb{N}$, the associated spectral problem obtained from linearizing about a given
wave $P_{a,b}$ is that for an ordinary differential equation with periodic coefficients.  As such,
it is an easy calculation to see that the spectrum of this linearized operator, considered as an operator on $L^2(\RM)$
is purely essential and agrees with the continuous spectrum.  In this classical case, a common method
for analyzing the essential spectrum of the linearization is to use Floquet theory to provide a
continuous parameterization of the essential spectrum by the eigenvalues of an associated one-parameter
family of linear operators considered with periodic boundary conditions; see \cite{BrJ,BrJZ,HK}, for example.

When $\alpha>\frac{1}{2}$ is not an even natural number, however, the linearized operator associated with $P_{a,b}$ is nonlocal and hence
the classical Floquet theory does not apply.  Nevertheless, we find in the next section that
we can still reduce the spectral stability problem on $L^2(\RM)$ to a one-parameter family
of eigenvalue problems with periodic boundary conditions.  This characterization of the essential
spectrum is given by utilizing the Floquet-Bloch transform defined on $L^2(\RM)$.  Once this characterization
is established, we use an appropriate spectral perturbation theory to analyze the eigenvalues
of the associated one-parameter family of Bloch operators when $|(a,b)|\ll 1$

\subsection{Characterization of the Essential Spectrum}

To analyze the spectral stability of the small amplitude periodic waves obtained in the previous section,
we fix $\alpha>\frac{1}{2}$ and set $z=k_{a,b}x$ and $s=k_{a,b}t$ in \eqref{e:travkdv} to get
\[
v_s+(-|k_{a,b}|^\alpha\Lambda^\alpha v-v+v^{p+1})_z,
\]
where now $\Lambda$ acts on $\RM$ in the $z$-variable.  Linearizing about $P_{a,b}$ and considering solutions
of the linearized equation of the form $v(z,t)=e^{\lambda t}v(z)$, with $\lambda\in\CM$ and $v(\cdot)\in L^2(\RM)$,
leads to the spectral problem
\[
\mathcal{M}_{a,b}v:=\partial_z\mathcal{L}_{a,b}v=\lambda v
\]
considered on $H^{\alpha+1}(\RM)$, where here $\lambda$ denotes the spectral parameter and
\[
\mathcal{L}_{a,b}:=|k_{a,b}|^\alpha\Lambda^\alpha+1-(p+1)P_{a,b}^p
\]
is considered as a closed, densely defined operator acting on $L^2(\RM)$.  In this case, spectral stability
is defined by the condition that the operator $\mathcal{M}_{a,b}$ have no spectrum in the open right half 
plane.  Notice, however, that the Hamiltonian form of the spectral problem\footnote{More precisely, the fact that
the spectral problem takes the form $J\mathcal{L}v=\lambda v$ where $v$ belongs to some Hilbert space $X$ and where 
$J$ is skew symmetric and $\mathcal{L}$ is self adjoint when acting on $X$.} implies the spectrum is symmetric with
respect to the real and imaginary axes, and hence spectral stability is equivalent with all the spectrum
of the operator $\mathcal{M}_{a,b}$ being confined to the imaginary axis; we will discuss this more in the next section.

In order to analyze the spectrum of the linear operator $\mathcal{M}_{a,b}$ acting on $L^2(\RM)$, we recall
now some 
facts about the Floquet-Bloch decomposition of $L^2(\RM)$; as mentioned
above, the standard Floquet theory does not apply for $\alpha\notin 2\mathbb{N}$ since the spectral problem
for $\mathcal{M}_{a,b}$ is not in the form of an ordinary differential equation.  To this end, 
notice
that given any $v\in L^2(\RM)$ we can express $v$ in terms of its inverse Bloch representation as
\[
v(x)=\int_{-1/2}^{1/2}e^{i\xi x}\check{v}(\xi,x)d\xi
\]
where $\check{v}(\xi,x):=\sum_{k\in\ZM}e^{ikx}\hat{v}(\xi+k)$ are $2\pi$-periodic functions of $x$ and where
$\hat{v}(\omega):=\frac{1}{2\pi}\int_{\RM}e^{-i\omega z}v(z)dz$ denotes the standard Fourier transform of $v$.  Indeed,
the above formulas may be easily checked on the Schwartz class by grouping frequencies which differ by one in the standard Fourier
transform representation of $v$:
\[
 v(z)
   =\sum_{j\in\ZM}\int_{-1/2}^{1/2}e^{i(k+j)z}\hat{v}(k+j)dk=\int_{-1/2}^{1/2}e^{ikz}\check{v}(\xi,z)dz.
\]
The Bloch transform $\mathcal{B}:L^2(\RM)\to L^2([-1/2,1/2);L^2(\RM/2\pi\ZM))$ given
by $\mathcal{B}(v)(\xi,x):=\check{v}(\xi,x)$ is then well defined, bijective, and continuous; in fact,
using the classical Parseval theorem we find for all $v\in L^2(\RM)$ that
\begin{equation}\label{par}
\|v\|_{L^2(\RM)}^2=2\pi\int_{-1/2}^{1/2}\int_{0}^{2\pi}\left|\mathcal{B}(v)(\xi,z)\right|^2dz~d\xi
\end{equation}
so that the rescaled Bloch transform $\sqrt{2\pi}\mathcal{B}$ is an isometry on $L^2(\RM)$\footnote{This is in fact a special case of a more general class of generalized
Hausdorff-Young type inequalities, following from interpolating \eqref{par} with the triangle inequality, that are known to be satisfied by the Bloch transform; see \cite{JZ} for details.}.

Furthermore, we find given
any $v\in L^2(\RM)$ that
\[
\mathcal{B}\left(\mathcal{M}_{a,b}v\right)(\xi,x)=\mathcal{M}_{a,b,\xi}\left(\check{v}(\xi,\cdot)\right)(x)
\]
where $\mathcal{M}_{a,b,\xi}:H^{\alpha+1}(\RM/2\pi\ZM)\subset L^2(\RM/2\pi\ZM)\to L^2(\RM/2\pi\ZM)$ is a defined
by 
\[
\mathcal{M}_{a,b,\xi}:=e^{-i\xi x}\mathcal{M}_{a,b}e^{i\xi x},\quad\xi\in[-1/2,1/2).
\]
The operators $\mathcal{M}_{a,b,\xi}$ are called the Bloch operators associated to $\mathcal{M}_{a,b}$,
and from above they
correspond to operator-valued symbols of $\mathcal{M}_{a,b}$ under $\mathcal{B}$
acting on $L^2(\RM/2\pi\ZM)$.  The next result relates the spectrum of $\mathcal{M}_{a,b}$ acting on $L^2(\RM)$
to that of the corresponding one-parameter family of Bloch operators $\left\{\mathcal{M}_{a,b,\xi}\right\}_{\xi\in[-1/2,1/2)}$,
providing us with a Floquet-like theory that is suitable for our needs.

\begin{proposition}\label{bloch}
Consider the operator $\mathcal{M}_{a,b}$ acting on $L^2(\RM)$ with domain $H^{\alpha+1}(\RM)$ and the associated Bloch operators
$\left\{\mathcal{M}_{a,b,\xi}\right\}_{\xi\in[-1/2,1/2)}$ acting on $L^2(\RM/2\pi\ZM)$ with domain
$H^{\alpha+1}(\RM/2\pi\ZM)$.  Then for any $\lambda\in\CM$, the following statements are equivalent:
\begin{itemize}
\item[(i)] $\lambda$ belongs to the spectrum of the closed operator $\mathcal{M}_{a,b}$ acting on $L^2(\RM)$.
\item[(ii)] There exists a $\xi\in[-1/2,1/2)$ such that $\lambda$ belongs to the spectrum of the closed
operator $\mathcal{M}_{a,b,\xi}$ acting on $L^2(\RM/2\pi\ZM)$.
\item[(iii)] There exists a nonzero function $V\in L^2(\RM/2\pi\ZM)$ of the form $V(z)=e^{i\xi z}v(z)$
for some $\xi\in[-1/2,1/2)$ and $v\in H^{\alpha+1}(\RM/2\pi\ZM)$ such that $\left(\mathcal{M}_{a,b}-\lambda{\bf I}\right)V=0$.
\end{itemize}
\end{proposition}

The proof of Proposition \ref{bloch} is contained in Appendix \ref{a:bloch}.  Using this result, it follows that
\[
\sigma_{L^2(\RM)}\left(\mathcal{M}_{a,b}\right)=
     \bigcup_{\xi\in[-1/2,1/2)}\sigma_{L^2_{\rm per}([0,2\pi])}\left(\mathcal{M}_{a,b,\xi}\right)
\]
so that, in particular, the Bloch transform provides a continuous parametrization of the essential
spectrum of the operator $\mathcal{M}_{a,b}$ by the discrete spectrum of a one-parameter family
of Bloch operators.
As a result, rather than analyzing the essential spectrum of the operator $\mathcal{M}_{a,b}$ directly, we can
instead choose to study the 
point spectrum of the operators $\mathcal{M}_{a,b,\xi}$ for each $\xi\in[-1/2,1/2)$.
This investigation is the subject of the next section.

\subsection{Analysis of the Unperturbed Operators}
We begin by considering the stability of the constant state $P_{0,0}=Q_0=1$.  Later, we will treat the linearized operators
about the nearby periodic solutions as small perturbations of those linearized about $Q_0$.  Indeed,
it is straightforward to establish the estimate 
\[
\|\mathcal{M}_{a,b,\xi}-\mathcal{M}_{0,0,\xi}\|=\mathcal{O}(|a|+|b|)
\]
as $(a,b)\to(0,0)$ uniformly in the Bloch parameter $\xi\in[-1/2,1/2)$.  A standard perturbation
argument then guarantees the spectrum of $\mathcal{M}_{a,b,\xi}$ and $\mathcal{M}_{0,0,\xi}$ stay close
for sufficiently small $(a,b)$.  More precisely, we have the following result.

\begin{lemma}
Let $p\geq 1$ and $\alpha>\frac{1}{2}$.  For any $\delta>0$ there exists an $\eps>0$ such that, for any $\xi\in[-1/2,1/2)$ and
any $(a,b)\in\RM^2$ with $\|(a,b)\|\leq\eps$, the spectrum of $\mathcal{M}_{a,b,\xi}$ satisfies
\[
\sigma\left(\mathcal{M}_{a,b,\xi}\right)\subset\left\{\lambda\in\CM:{\rm dist}\left(\lambda,\sigma\left(\mathcal{M}_{0,0,\xi}\right)\right)<\delta\right\}.
\]
\end{lemma}

We now analyze the spectrum of $\mathcal{M}_{0,0,\xi}$ posed on $L^2(\RM/2\pi\ZM)$.  
Since $\mathcal{M}_{0,0,\xi}$ has constant coefficients, we find by a straightforward Fourier analysis argument
that
\[
\sigma(\mathcal{M}_{0,0,\xi})=\left\{\lambda=i\omega_{n,\xi}:n\in\ZM\right\}\subset i\RM
\]
for each fixed $\xi\in[-1/2,1/2)$, where the $\omega_{n,\xi}$ are determined by the linear dispersion
relation 
\begin{equation}\label{e:lindisp}
\omega(k):=kp\left(|k|^\alpha-1\right)
\end{equation}
through $\omega_{n,\xi}:=\omega(n+\xi)$.  Notice that every $\lambda\in\sigma(\mathcal{M}_{0,0,\xi})$
is a semi-simple eigenvalue with algebraic and geometric multiplicity given by the number
of distinct $n\in\mathbb{Z}$ such that $\lambda=i\omega_{n,\xi}$ with associated
eigenfunction $e_n:=e^{inz}$.

To study the behavior of these eigenvalues for small $(a,b)\in\RM^2$, notice that the spectrum of the
operator $\mathcal{M}_{a,b}$ is symmetric with respect to both the real and imaginary axis.  Indeed,
since the coefficients of $\mathcal{M}_{a,b}$ are real valued it follows that its spectrum
is symmetric with respect to the real axis.  In terms of the Bloch operators this implies
that $\sigma(\mathcal{M}_{a,b,\xi})=\overline{\sigma(\mathcal{M}_{a,b,-\xi})}$.  Furthermore, noting
that $\mathcal{M}_{a,b}$ anti-commutes with the isometry $S:L^2(\RM)\to L^2(\RM)$ given by
\[
Sv(z)=v(-z)
\]
we find that the spectrum of $\mathcal{M}_{a,b}$ is symmetric with respect to the origin.  It follows
that the Bloch operators satisfy $\mathcal{M}_{a,b,\xi}S=-S\mathcal{M}_{a,b,\xi}$ so that
$\sigma(\mathcal{M}_{a,b,\xi})=-\sigma(\mathcal{M}_{a,b,-\xi})$.
Finally, recalling that $k_{a,b}$ is even
in $a$ and that $P_{a,b}(z+\pi)=P_{-a,b}(z)$ for all $|(a,b)|\ll 1$ we see that
$\sigma(\mathcal{M}_{a,b})=\sigma(\mathcal{M}_{-a,b})$ and that $\sigma(\mathcal{M}_{a,b,\xi})=\sigma(\mathcal{M}_{-a,b,\xi})$.
As a result, we find that the spectrum of a given Bloch operator $\mathcal{M}_{a,b,\xi}$ is symmetric with
respect to the imaginary axis.  It follows then that when eigenvalues of $\mathcal{M}_{a,b,\xi}$ 
bifurcate from the imaginary axis they must bifurcate in pairs resulting from collisions of eigenvalues
on the imaginary axis.  A well known result from the study of Hamiltonian systems tells us that when
two purely imaginary eigenvalues collide, the collision will not result in a pair of eigenvalues
bifurcating from the imaginary axis provided both eigenvalues have the same 
Krein signature; see \cite{YS}.

We now consider the location of the eigenvalues more carefully, in particular watching for sets of eigenvalues
that collide for a fixed $\xi$.  
Notice by the symmetry property $\sigma\left(\mathcal{M}_{a,b,\xi}\right)=\overline{\sigma\left(\mathcal{M}_{a,b,-\xi}\right)}$
we may restrict our consideration to Bloch frequencies $\xi\in[0,1/2]$.
Now, when $\xi=0$ we find that
\[
\omega_{-1,0}=\omega_{0,0}=\omega_{1,0}=0
\]
and
\[
\ldots<\omega_{-3,0}<\omega_{-2,0}<0<\omega_{2,0}<\omega_{3,0}<\ldots.
\]
Furthermore, 
for $\xi\in[0,1/2]$ we have
\[
\omega_{n,\xi}\subset\left(-\infty,-\frac{3p}{2}\left(\left(\frac{3}{2}\right)^\alpha-1\right)\right]
    \cup\left[2p(2^\alpha-1),\infty\right),\quad |n|\geq 2
\]
and that
\[
\omega_{n,\xi}\subset\left[\frac{p}{2}\left(\frac{1}{2^\alpha}-1\right),\frac{3p}{2}\left(\left(\frac{3}{2}\right)^\alpha-1\right)\right],\quad|n|\leq 1,
\]
which naturally provides us with a spectral decomposition
\[
\sigma\left(\mathcal{M}_{0,0,\xi}\right)=\sigma_1\left(\mathcal{M}_{0,0,\xi}\right)\cup\sigma_2\left(\mathcal{M}_{0,0,\xi}\right)
\]
for $\mathcal{M}_{0,0,\xi}$ with
\[
\left\{\begin{aligned}
\sigma_1\left(\mathcal{M}_{0,0,\xi}\right)&=\left\{\lambda\in\CM:\lambda=i\omega_{n,\xi}\textrm{ for some }|n|\geq 2\right\}\\
\sigma_2\left(\mathcal{M}_{0,0,\xi}\right)&=\left\{i\omega_{-1,\xi},i\omega_{0,\xi},i\omega_{1,\xi}\right\},
\end{aligned}\right.
\]
with the property that for any $v$ in the infinite dimensional spectral subspace associated with $\sigma_1\left(\mathcal{M}_{0,0,\xi}\right)$ 
satisfies
\[
\left<\mathcal{L}_{0,0,\xi}v,v\right>\geq\left(\left(\frac{3}{2}\right)^\alpha-1\right)p\|v\|^2
\]
uniformly for $\xi\in[0,1/2]$, where 
\[
\mathcal{L}_{0,0,\xi}:=e^{-i\xi x}\mathcal{L}_{0,0}e^{i\xi x}
\]
is considered on $L^2(\RM/2\pi\ZM)$; in particular, it follows that all eigenvalues in $\sigma_1(\mathcal{M}_{0,0,\xi})$
have positive Krein signature for $\xi\in[0,1/2]$.
By a standard perturbation argument, 
we find that the above properties persist for sufficiently small $a$ and $b$.  More precisely, one has
that for $a$ and $b$ sufficiently small we have a spectral decomposition
\[
\sigma\left(\mathcal{M}_{a,b,\xi}\right)=\sigma_1\left(\mathcal{M}_{a,b,\xi}\right)\cup\sigma_2\left(\mathcal{M}_{a,b,\xi}\right)
\]
such that 
\[
\sigma_1\left(\mathcal{M}_{a,b,\xi}\right)\cap\sigma_2\left(\mathcal{M}_{a,b,\xi}\right)=\emptyset,
\]
for $|(a,b)|\ll 1$ and $\xi\in[0,1/2]$, where
the spectral subspace associated with $\sigma_2\left(\mathcal{M}_{a,b,\xi}\right)$ is three dimensional,
and the infinitely many eigenvalues associated with $\sigma_1\left(\mathcal{M}_{a,b,\xi}\right)$
all have positive Krein signature.  Notice this latter property implies that all the eigenvalues
in $\sigma_1\left(\mathcal{M}_{a,b,\xi}\right)$ are purely imaginary for $|(a,b)|$ sufficiently small.

It now remains to determine the location of the eigenvalues in $\sigma_2\left(\mathcal{M}_{a,b,\xi}\right)$.
These are smooth continuations for small $a,b\in\RM$ of the eigenvalues $i\omega_{-1,\xi}$, $i\omega_{0,\xi}$, and
$i\omega_{1,\xi}$ of $\mathcal{M}_{a,b,\xi}$.  First, notice that two such eigenvalues collide if and only
if $\xi=0$, when $\omega_{\pm 1,0}=\omega_{0,0}=0$.  Furthermore, for any $\xi_0\in(0,1/2)$ there exists
a constant $c_0>0$ such that
\[
\left|\omega_{j,\xi}-\omega_{k,\xi}\right|\geq c_0 p,\quad j,k\in\{-1,0,1\},~j\neq k,~\xi\in[\xi_0,1/2].
\]
Consequently, these eigenvalues will remain simple and distinct under small perturbations (for sufficiently
small $a$ and $b$) for all $\xi\in[0,1/2]$.  In particular, it follows that for any $\xi_0\in(0,1/2]$ we have
\[
\sigma_2\left(\mathcal{M}_{a,b,\xi}\right)\subset i\RM,\quad\xi\in[\xi_0,1/2]
\]
for sufficiently small $a$ and $b$.  Thus, we have reduced the problem to locating the eigenvalues
$i\omega_{\pm 1,\xi}$ and $i\omega_{0,\xi}$ for $a$, $b$, and $\xi$ small.

\subsection{Location of the Critical Bloch Spectrum}\label{s:crit}

To this end, for $|(a,b,\xi)|\ll 1$ we compute a suitable basis $\{\eta_j(z;a,b,\xi)\}_{j=0,1,2}$ for the  three dimensional spectral subspace associated
to $\sigma_2(\mathcal{M}_{a,b,\xi})$ and compute the $3\times 3$ matrices
\[
\mathcal{B}_{a,b,\xi}:=\left[\left(\left<\frac{\eta_j(z;a,b,\xi)}{\left<\eta_j(z;a,b,\xi),\eta_j(z;\a,b,\xi)\right>},
      \mathcal{M}_{a,b,\xi}\eta_k(z;a,b,\xi)\right>\right)\right]_{j,k=0,1,2}
\]
and 
\[
I_{a,b,\xi}:=\left[\left(\frac{\left<\eta_j(z;a,b,\xi),\eta_k(z;a,b,\xi)\right>}{\left<\eta_j(z;a,b,\xi),\eta_j(z;\a,b,\xi)\right>}\right)\right]_{j,k=0,1,2},
\]
representing, respectively, the actions of $\mathcal{M}_{a,b,\xi}$ and the identity operator
on this space.  First, at $a=b=0$ we have that the operator $\mathcal{M}_{0,0,\xi}$
has constant coefficients and a basis for the critical spectral subspace in this case
is spanned by the functions $1$ and $e^{\pm iz}$ associated with the eigenvalues
$i\omega_{0,\xi}$ and $i\omega_{\pm 1,\xi}$, respectively.  Working instead with the real
basis 
\[
\eta_0(z;0,0,\xi)=\cos(z),~~\eta_1(z;0,0,\xi)=\sin(z),~~\eta_2(z;0,0,\xi)=1
\]
a direct calculation yields
\[
\mathcal{B}_{0,0,\xi}=\left(\begin{array}{ccc}
                            \frac{i}{2}\left(\omega_{1,\xi}+\omega_{-1,\xi}\right) & \frac{1}{2}\left(\omega_{1,\xi}-\omega_{-1,\xi}\right) & 0\\
                            -\frac{1}{2}\left(\omega_{1,\xi}-\omega_{-1,\xi}\right) & \frac{i}{2}\left(\omega_{1,\xi}+\omega_{-1,\xi}\right) & 0\\
                            0 & 0 & i\omega_{0,\xi}
                            \end{array}\right)
\]
valid for any $\xi\in[0,1/2]$.  Furthermore, by the definition of $k_{0,b}^\alpha$ given in Theorem \ref{t:existence}, we see that
the above functions $\eta_j$ also form a basis for the null space of the operator $\mathcal{M}_{0,b,\xi}$, from which it follows
that we may take
\[
I_{0,b,\xi}=\left(\begin{array}{ccc}1&0&0\\0&1&0\\0&0&1\end{array}\right)
\]
for all $b$ and $\xi$ sufficiently small.  Here and throughout, we are using $\left<f,g\right>=\int_0^{2\pi}\overline{f(z)}g(z)dz$.

Next, at $\xi=0$ we claim that $\lambda=0$ is an eigenvalue of $\mathcal{M}_{a,b,0}$ with algebraic multiplicity three
and geometric multiplicity two.  Indeed, notice that since \eqref{e:kdv} is invariant with respect to spatial translations
it follows that
\[
\mathcal{M}_{a,b,0}\partial_zP_{a,b}(z)=0
\]
so that $\lambda=0$ is indeed an eigenvalue of $\mathcal{M}_{a,b,0}$.  Furthermore, differentiating the profile equation \eqref{e:profileeqn}
with respect to the parameters $a$ and $b$ yield
\begin{align*}
\mathcal{M}_{a,b,0}\partial_aP_{a,b}(z)&=-\partial_a\left(|k_{a,b}|^\alpha\right)\Lambda^\alpha\partial_zP_{a,b}\\
\mathcal{M}_{a,b,0}\partial_bP_{a,b}(z)&=-\partial_b\left(|k_{a,b}|^\alpha\right)\Lambda^\alpha\partial_zP_{a,b}
\end{align*}
so that, in particular, we find
\[
\mathcal{M}_{a,b,0}\left(\partial_b\left(|k_{a,b}|^\alpha\right)\partial_aP_{a,b}(z)-
          \partial_a\left(|k_{a,b}|^\alpha\right)\partial_bP_{a,b}(z)\right)=0,
\]
giving a second function in the kernel of $\mathcal{M}_{a,b,0}$.  Finally, by a straight forward computation,
using \eqref{e:profileeqn} and the fact that $\mathcal{L}_{a,b,0}\partial_zP_{a,b}=0$, we find 
\[
\mathcal{M}_{a,b,0}P_{a,b}(z)=-p\left(|k_{a,b}|^\alpha\Lambda^\alpha +1\right)\partial_zP_{a,b}
\]
so that
\[
\mathcal{M}_{a,b,0}\left(\partial_b\left(|k_{a,b}|^\alpha\right)P_{a,b}-p|k_{a,b}|^\alpha\partial_bP_{a,b}\right)
     =-p\partial_b\left(|k_{a,b}|^\alpha\right)\partial_zP_{a,b},
\]
giving a $2\pi$-periodic function in the generalized kernel of $\mathcal{M}_{a,b,0}$ ascending above the translation mode $\partial_zP_{a,b}$.
In particular, this shows that for all values of $a$ and $b$ sufficiently small, 
the three eigenvalues $i\omega_{0,\xi}$ and $i\omega_{\pm 1,\xi}$ all vanish when $\xi=0$, and that the eigenvalue
$\lambda=0$ indeed corresponds to an algebraically triple and geometrically double eigenvalue of the operator 
$\mathcal{M}_{a,b,0}$ acting on $L^2(\RM/2\pi\ZM)$.  

Using the expansions in Lemma \ref{l:smallampexpand} we now 
obtain a basis for the critical eigenspace which is compatible with the basis found at $a=b=0$.  In particular,
we take
\begin{align*}
\eta_0(z;a,b,0)&=\frac{1}{p+1}\left(\partial_b\left(|k_{a,b}|^\alpha\right)\partial_aP_{a,b}(z)-
          \partial_a\left(|k_{a,b}|^\alpha\right)\partial_bP_{a,b}(z)\right)\\
     &=\cos(z)-\frac{2+2^\alpha(p-1)}{4(2^\alpha-1)}a+\frac{p+1}{6}\cos(2z)a-\frac{2}{p}\cos(z)b+\mathcal{O}(a^2+b^2)\\
\eta_1(z;a,b,0)&=-\frac{1}{a}\partial_zP_{a,b}=\sin(z)+\frac{p+1}{2(2^\alpha-1)}\sin(2z)a+\mathcal{O}(a^2+b^2)\\
\eta_2(z;a,b,0)&=\partial_b\left(|k_{a,b}|^\alpha\right)P_{a,b}-p|k_{a,b}|^\alpha\partial_bP_{a,b}\\
     &=1+(p+1)\cos(z)a-\frac{p+1}{p}b+\mathcal{O}(a^2+b^2).
\end{align*}
In this basis, a straightforward calculation yields
\[                          
I_{a,b,0}=\left(\begin{array}{ccc}
1 & 0 & \frac{2^\alpha(p+3)-2(p+2)}{2(2^\alpha-1)}a\\
0 & 1 & 0\\
\frac{2^\alpha(p+3)-2(p+2)}{4(2^\alpha-1)}a & 0 & 1
\end{array}\right)+\mathcal{O}(a^2+b^2)
\]
and
\[
\mathcal{B}_{a,b,0}=\left(\begin{array}{ccc}
                          0 & 0 & 0\\
                          0 & 0 & \sigma_{a,b}\\
                          0 & 0 & 0
                          \end{array}\right),
\]
where
\begin{align*}
\sigma_{a,b}&=\frac{\left<\eta_1(z;a,b,0),\mathcal{M}_{a,b,0}\eta_2(z;a,b,0)\right>}{\left<\eta_1(z;a,b,0),\eta_1(z;a,b,0)\right>}=pa\partial_b\left(|k_{a,b}|^\alpha\right)\\
&=p(p+1)a-2(p+1)ab+\mathcal{O}(|a|(a^2+b^2)).
\end{align*}

In order to analyze the breaking of the algebraically triple eigenvalue at $\lambda=0$ of the operator
$\mathcal{M}_{a,b,\xi}$ for $|\xi|\ll 1$, we need the following result which allows us to expand
the nonlocal dispersive operator $\partial_x\Lambda^\alpha$ with respect to the Bloch frequency $\xi$.

\begin{lemma}\label{l:blochexpand}
Let $\alpha>0$ and $f\in H^{\alpha+1}(\RM/2\pi\ZM)$ be fixed.  Then for all $|\xi|\ll 1$, we have
\begin{align*}
&e^{-i\xi x}\partial_x\Lambda^\alpha e^{ikx}f(x)=
\partial_x\Lambda^\alpha f(x)+i\xi|\xi|^\alpha\hat{f}(0)\\
&\qquad\qquad+\left(\sum_{l=1}^\infty i{\alpha+1\choose 2l-1}\Lambda^{\alpha-2(l-1)}\xi^{2l-1}\right)\mathcal{P}f(x)\\
&\qquad\qquad\qquad+\left(\sum_{r=1}^\infty {\alpha+1\choose 2r}\partial_x\Lambda^{\alpha-2r}\xi^{2r}\right)\mathcal{P}f(x),
\end{align*}
where ${m\choose n}:=\frac{m(m-1)(m-2)\ldots(m-n+1)}{n!}$ denotes the generalized Binomial coefficient,
defined for $m\in\CM$ and $n\in\mathbb{N}$, and $\mathcal{P}$ denotes the orthogonal projection of $H^{\alpha+1}(\RM/2\pi\ZM)$
onto the subspace of mean-zero functions.
\end{lemma}

\begin{proof}
Using the Fourier series representation of $f$ we find
\[
e^{-i\xi x}\partial_x\Lambda^\alpha e^{i\xi x}f(x)=\sum_{k\in\ZM}i(k+\xi)|k+\xi|^\alpha e^{ikx}\hat{f}(k).
\]
Since $|\xi|\ll 1$, we have $|k+\xi|=|k|+\sgn(k)\xi$ so that, for $k\neq 0$, the term $|k+\xi|^\alpha$ may be expanded
using Newton's Binomial series as 
\[
|k+\xi|^\alpha=\sum_{m=0}^\infty {\alpha\choose m} |k|^{\alpha-m}\sgn(k)^m\xi^m,
\]
and hence
\[
(k+\xi)|k+\xi|^\alpha=|k|^\alpha k+\sum_{m=1}^\infty\sgn(k)^{m-1}\left[{\alpha\choose m}+{\alpha\choose m-1}\right]
|k|^{\alpha-(m-1)}\xi^m,
\]
valid for $k\neq 0$.
If $m=2l-1$ for some $l\in\mathbb{N}$, then 
\[
\sgn(k)^{m-1}|k|^{\alpha-(m-1)}=|k|^{\alpha-2(l-1)}
\]
while if $m=2r$ for some $r\in\mathbb{N}$, we have
\[
\sgn(k)^{m-1}|k|^{\alpha-(m-1)}=k|k|^{\alpha-2r}.
\]
Using Pascal's rule ${m\choose n}+{m\choose n-1}={m+1\choose n}$, valid for all $m\in\CM$ and $n\in\mathbb{N}$,
we find for $k\neq 0$
\begin{align*}
(k+\xi)|k+\xi|^\alpha&=k|k|^\alpha+\sum_{l=1}^\infty{\alpha+1\choose 2l-1}|k|^{\alpha-2(l-1)}\xi^{2l-1}\\
&\quad+\sum_{r=1}^\infty{\alpha+1\choose 2r}k|k|^{\alpha-2r}\xi^{2r},
\end{align*}
from which the claimed expansion follows.
\end{proof}

\begin{remark}\label{r:smooth}
Notice that, for a given $\alpha>0$ the smoothness of the operator $e^{-i\xi x}\partial_x\Lambda^\alpha e^{i\xi x}$
in $\xi$ near $\xi=0$ is determined by the smoothness of the term $\xi|\xi|^\alpha$ near the origin.  In particular,
for $\alpha\in 2\mathbb{N}$ one obtains an analytic expansion, while for positive $\alpha\notin 2\mathbb{N}$ one obtains
only a $C^{\lfloor \alpha\rfloor+1}$ expansion.  It is expected then that for every $\alpha>0$ the spectrum
of $\mathcal{M}_{a,b,\xi}$ bifurcating from the $(\lambda,\xi)=(0,0)$ will be at least $C^1$
in a neighborhood of $\xi=0$.
\end{remark}

From Lemma \ref{bloch}, we can expand for  $|(a,b,\xi)|\ll 1$ the operator $\mathcal{M}_{a,b,\xi}$ to
any desired order in $\xi$.  For our purposes, it is sufficient to identify the term of first
order in $\xi$, which is
\begin{equation}\label{order1}
\frac{\partial}{\partial\xi}\mathcal{M}_{a,b,\xi}\Big{|}_{\xi=0}=
i\left((\alpha+1)k_{a,b}^\alpha\Lambda^\alpha+1-(p+1)P_{a,b}^p\right).
\end{equation}
Notice that outside the small amplitude regime $|(a,b)|\ll 1$ one would have to consider higher order
terms in the expansion as well.  Indeed, in the classical case $\alpha=2$, corresponding to the KdV equation,
one typically must expand the Bloch operator to \emph{second} order in $\xi$; see \cite{BrJZ}.
For large amplitude periodic waves then, Remark \ref{r:smooth} then seems to indicate one must take 
care to consider the cases $\alpha\geq 1$ and $\alpha\in(0,1)$ separately.  In our case, however, we will see below
that the restriction to small amplitude solutions requires only an expansion to first order, hence bypassing
this added difficulty.

Continuing, we note that the basis $\{\eta_j(z;a,b,0)\}_{j=0,1,2}$ defined above can be extended to a basis 
of the three-dimensional eigenspace bifurcating from the generalized kernel of $\mathcal{M}_{a,b,0}$ 
for sufficiently small $a,b$ and $\xi$.  This provides us with an expansion of the form 
\begin{equation}\label{matrixexpand}
\mathcal{B}_{a,b,\xi}=\mathcal{B}_{0,0,\xi}+\mathcal{B}_{a,b,0}+a\xi\mathcal{B}_1+b\xi\mathcal{B}_2
    +\mathcal{O}\left(|\xi|(a^2+b^2)+|\xi|^{\min(2,\alpha+1)}(a+b)\right).
\end{equation}
In order to determine the three eigenvalues of $\mathcal{B}_{a,b,\xi}$, we consider the characteristic
polynomial
\[
D(\lambda;a,b,\xi)=-\det\left(\mathcal{B}_{a,b,\xi}-\lambda I_{a,b,\xi}\right)=\lambda^3+c_2\lambda^2+c_1\lambda+c_0
\]
where the coefficient functions $c_j=c_j(a,b,\xi)$, defined for $|(a,b,\xi)|\ll 1$, 
depend smoothly on the parameters $a,b$ and are $C^1$ in $\xi$.  Analyzing the
dependence of the $c_j$ on $a$, $b$, and $\xi$ more carefully, notice that
since the spectrum of $\mathcal{M}_{a,b,\xi}$ is symmetric with respect to the imaginary
axis, considering the coefficients of $D$ as symmetric functions of its roots, we see that 
the $c_2$ and $c_0$ must be purely imaginary whereas $c_1$ must be real.  
Furthermore, since 
\[
\sigma\left(\mathcal{M}_{a,b,\xi}\right)=\sigma\left(\mathcal{M}_{-a,b,\xi}\right),\quad
\sigma\left(\mathcal{M}_{a,b,\xi}\right)=\overline{\sigma\left(\mathcal{M}_{a,b,-\xi}\right)},
\]
it follows that the coefficients $c_j$ must all be even in $a$ and that $c_2$ and $c_0$ are odd
in $\xi$ while $c_1$ is even in $\xi$.  Together with the expansion of $\mathcal{B}_{a,b,\xi}$ above,
these properties imply that the polynomial $D$ can be expressed as
\[
D(\lambda;a,b,\xi)=-(ip\xi)^3\det\left(\frac{1}{ip\xi}\mathcal{B}_{a,b,\xi}-\frac{\lambda}{ip\xi}I_{a,b,\xi}\right)
\]
so that the critical eigenvalues of $\mathcal{M}_{a,b,\xi}$ bifurcating from the $(\lambda,\xi)=(0,0)$
state are of the form $\lambda=ip\xi X$ where $X$ is a root of the cubic polynomial
\[
Q(X;a,b,\xi)=-\det\left(\frac{1}{ip\xi}\mathcal{B}_{a,b,\xi}-XI_{a,b,\xi}\right)=X^3+d_2X^2+d_1X+d_0,
\]
where the coefficients $d_j$ are real-valued and even in $a$ and $\xi$.  To determine whether the roots
of $Q$, and hence the three critical eigenvalues, lie on the imaginary axis or not we consider the discriminant
\[
\Delta_{a,b,\xi}=18d_2d_1d_0+d_2^2d_1^2-4d_2^3d_0-4d_1^3-27d_0^2,
\]
which here, by the symmetry properties of the coefficients $d_j$, can be expanded near for $|(a,b,\xi)|\ll 1$ as
\begin{equation}\label{discexpand}
\Delta_{a,b,\xi}=\Delta_{0,b,\xi}+\gamma a^2+\mathcal{O}\left((a^2\left(a^2+|b|+\xi^\delta\right)\right)
\end{equation}
for some appropriate $\delta>0$, chosen independently of $\xi$, and some constant $\gamma\in\RM$ still to be determined.
In particular, the polynomial $Q$ will have three real roots, corresponding to stability, when $\Delta_{a,b,\xi}>$ while
it will have one real and two complex-conjugate roots, corresponding to instability, when $\Delta_{a,b,\xi}<0$.

Notice that the operator $\mathcal{M}_{0,b,\xi}$ has constant coefficients and is clearly
skew-adjoint.  Consequently, the spectrum of $\mathcal{M}_{0,b,\xi}$ must lie on $i\RM$ so that,
in particular, $\Delta_{0,b,\xi}\geq 0$.  More precisely, using \eqref{e:lindisp} these eigenvalues 
are given explicitly by
\[
i\xi|k_{0,b}|^\alpha\left(|\xi|^\alpha-1\right),\quad i(\pm 1+\xi)|k_{0,b}|^\alpha\left(|\pm 1+\xi|^\alpha-1\right).
\]
It follows that
\[
\Delta_{0,b,\xi}=\frac{|k_{0,b}|^{6\alpha}}{p^6\xi^6}\left(A_1(\xi;\alpha)A_1(-\xi;\alpha)A_2(\xi;\alpha)\right)^2,
\]
where, for $|\xi|\ll 1$, we have the expansions
\begin{align*}
A_1(\xi;\alpha)&=-1+(1+\xi)^{\alpha+1}-\xi|\xi|^\alpha=(\alpha+1)\xi+\mathcal{O}(|\xi|^{1+\eps})\\
A_2(\xi;\alpha)&=-2+(1-\xi)^{\alpha+1}+(1+\xi)^{\alpha+1}=\alpha(\alpha+1)\xi^2+\mathcal{O}(|\xi|^{2+\eps})
\end{align*}
for some appropriate $\eps>0$ chosen independently of $\xi$.  Thus, for $|\xi|\ll 1$ we have
\[
\Delta_{0,b,\xi}=\frac{|k_{0,b}|^{6\alpha}\alpha^2(\alpha+1)^6\xi^2}{p^6}+\mathcal{O}(|\xi|^{2+\eps}),
\]
which is clearly positive for $\xi$ sufficiently small.

To compute the coefficient $\gamma$ in \eqref{discexpand}, notice it is enough to compute $\Delta_{a,0,0}$ to $\mathcal{O}(a^2)$.  
To this end, using \eqref{order1} we find that the matrix $\mathcal{B}_1$ in \eqref{matrixexpand} is given by\footnote{Notice that expliit
forms of the variations $\partial_\xi \eta_j|_{\xi=0}$ are not needed at this order in the calculation.}
\[
\mathcal{B}_1=i\left(\begin{array}{ccc}
                    0 & 0 & (\alpha-1)p(p+1)+\frac{p(2+2^\alpha(p-1)}{2(2^\alpha-1)}\\
                    0 & 0 & 0\\
                    \frac{(\alpha-1)p(p+1)}{2}+\frac{p(2+2^\alpha(p-1)}{4(2^\alpha-1)} & 0 & 0
                    \end{array}\right),
\]
which gives, using Mathematica \cite{Ma}, the expansion
\[
\Delta_{a,0,0}=\left(\frac{(p+1)\alpha(1+\alpha)^4\left[2^\alpha\left(4-(p-1)(\alpha-1)\right)-4-2(\alpha+p)\right]}{2(2^\alpha-1)}\right)a^2+\mathcal{O}(a^4).
\]
Defining
\begin{equation}\label{e:critpower}
p^*(\alpha):=\frac{2^\alpha(3+\alpha)-4-2\alpha}{2+2^\alpha(\alpha-1)},
\end{equation}
it follows that for $p>\max\left(1,p^*(\alpha)\right)$ we have $\Delta_{a,0,0}<0$ for $|a|\ll 1$, corresponding to instability.
In particular, noting that $p^*(\alpha)<1$ for $\alpha\in(1/2,1)$, we find instability for all $p\geq 1$ when $\alpha\in(1/2,1)$.
On the other hand, for $\alpha>1$ and $1\leq p<p^*(\alpha)$, we find that $\Delta_{a,0,0}>0$ for $|a|\ll1$ corresponding
to stability for sufficiently small $a$
Together, the above analysis establishes our main result.


\begin{theorem}\label{t:kdvmain}
Let $\alpha>\frac{1}{2}$ and $p\geq 1$ be fixed such that either $p\in\mathbb{N}$ or $p=\frac{m}{n}$ where $m$ and $n$ 
are even and odd integers, respectively.  Then the small amplitude periodic traveling wave
solutions $u_{a,b}=P_{a,b}(k_{a,b}\cdot)$ of \eqref{e:kdv} 
constructed in Theorem \ref{t:existence} for $|(a,b)|\ll 1$ are spectrally stable if $\alpha>1$ and $p<p^*(\alpha)$,
and are spectrally unstable if $\alpha\in(1/2,1)$ or if $\alpha>1$ and $p>p^*(\alpha)$, where here $p^*(\alpha)$
is defined in \eqref{e:critpower}.
\end{theorem}

\begin{remark}
As mentioned in the introduction, the spectral instability detected in Theorem \ref{t:kdvmain} is of sideband type,
and is hence not detectable if one restricts to perturbations with the same period
as the underlying wave.
\end{remark}

\begin{remark}\label{r:bo}
For the classical Benjamin-Ono equation, corresponding to $\alpha=p=1$, we find from $p^*(1)=1$ that $\Delta_{a,0,0}=\mathcal{O}(a^4)$ for $|a|\ll 1$.
In this case, the above analysis is insufficient to determine the stability/instability of the small periodic traveling wave solutions, and one
must hence carry the above calculations to higher orders.  We do not attempt this here, but consider this an interesting open question.
\end{remark}

As far as the author is aware, Theorem \ref{t:kdvmain} is the first
rigorous spectral stability/instability result in the periodic context for nonlocal KdV like equations of form \eqref{e:kdv}.
A plot of $p^*(\alpha)$ for $\alpha\geq 1$ is provided in Figure \ref{f:critp}.
Notice that when $\alpha=2$, corresponding to the classic gKdV equation \eqref{e:gkdv}, Theorem \ref{t:kdvmain} recovers
the result of Haragus and Kapitula in \cite{HK}, yielding spectral stability for $1\leq p<2$ and spectral instability
for $p>2$.  As one begins to decrease $\alpha$ from this classical case, the critical power $p^*(\alpha)$ 
decreases nonlinearly, in contrast to the solitary wave case where the critical power decreases linearly as $\alpha$ is decreased; see \cite{BSS}. 
This illustrates an interesting, but not necessarily  unexpected, difference between the solitary and periodic theories.
It is also interesting to note that $p^*(\alpha)<2\alpha$ for all $\alpha>\frac{1}{2}$, indicative of the fact that periodic traveling waves
are generally \emph{less} stable than their solitary wave counterparts.  This seems intuitively clear since the class of periodic traveling wave solutions
of \eqref{e:kdv} generally has a richer structure than that for solitary waves.  Also, the admissible classes of perturbations for periodic waves
is considerably larger: indeed, in the periodic setting one can consider perturbations with twice the fundamental period of the undelrying wave, a situation
that is not possible in the solitary wave theory, which more closely resembles a co-periodic (i.e. zero-Bloch frequency) stability analysis.

What seems striking about Theorem \ref{t:kdvmain} is the fact that the function $p^*(\alpha)$ attains a global maximum of approximately $2.19$ on $(1/2,\infty)$ at
a critical $\alpha^*\approx 2.7486$, and that for $\alpha>\alpha^*$ the function $p^*(\alpha)$ monotonically decreases
with limiting behavior
\[
\lim_{\alpha\to\infty}p^*(\alpha)=1.
\]
This suggests that for $p\in(1,p^*(\alpha^*))$, there exists a finite range of $\alpha$ such that the associated model
equation \eqref{e:kdv} admits spectrally stable small periodic  travleing waves.  Indeed, for $p=1$ this indicates that 
for all $\alpha>1$ the small amplitude periodic traveling wave solutions of \eqref{e:kdv}, as constructed in Theorem \ref{t:existence}, are 
spectrally stable to localized perturbations on the line.  For $p=2$, on the other hand, we see that \eqref{e:kdv} admits spectrally
stable small periodic traveling wave solutions provided $\alpha\in(2,4)$, and that for $\alpha\notin[2,4]$ no such spectrally
stable waves exist.  Furthermore, for $p>p^*(\alpha^*)$, the model equation \eqref{e:kdv} does not admit spectrally stable
small amplitude periodic traveling waves for any $\alpha>1/2$.  It is also important to note that $p^*(1)=1$, so that
our analysis is unable to determine the spectral stability of the small amplitude periodic traveling waves 
constructed in Theorem \ref{t:existence} in the case $\alpha=p=1$, i.e. our analysis is inconclusive
regarding the stability of such waves in the classical Benjamin-Ono equation.  It does, however, indicate that
all such small periodic waves are spectrally unstable in the generalized Benjamin-Ono equation, corresponding
to \eqref{e:kdv} with $\alpha=1$ and $p>1$.

Finally, recall that the solitary wave solutions of \eqref{e:kdv} are known to be spectrally (and nonlinearly) 
stable provided $p<2\alpha$ \cite{BSS}.  Hence, for any $p\geq 1$ it is possible to find a stable solitary wave solution
of \eqref{e:kdv} so long as $\alpha$ is sufficiently large.  In contrast, the analysis presented in this paper provides not only an
upper bound $p^*(\alpha^*)$ on $p$ for which models of the form \eqref{e:kdv} can admit spectrally stable small periodic traveling waves, but also
provides for each $p\in (1,p^*(\alpha^*))$ lower \emph{and} upper bounds on the dispersion parameter $\alpha$ for such stable
waves to exist.  It seems possible that this striking difference between the solitary and periodic wave cases is fundamentally due
to the nature of dispersion on the line versus ``dispersion" on the circle (corresponding to periodic boundary conditions).  
Indeed, notice that the sideband type instability detected by our analysis
can be seen, via the Bloch transform, as a \emph{periodic} instability in the space $L^2(\RM/2\pi n\ZM)$ for some $n\gg 1$.  
At this time, however, we do not attempt to provide a more detailed or rigorous account for the striking differences
between the periodic and solitary wave cases, and instead leave this as an interesting open problem.

\vspace{2mm}

{\bf Acknowledgements:} The author is grateful to Vera Mikyoung Hur, Atanas Stefonov, Mariana Haragus, and L. Miguel Rodrigues for valuable discussions at
various stages of this project.  He also gratefully acknowledges
support from the National Science Foundation under NSF grant DMS-1211183, and from the University of Kansas General Research
Fund allocation 2302278.  Finally, the author thanks the referees for their careful reading of the manuscript and for their many useful
suggestions and references.

\appendix

\section{Existence of Small Periodic Traveling Waves}\label{a:exist}

In this appendix, we utilize a Lyapunov-Schmidt reduction to establish the existence and basic regularity
properties of small amplitude periodic traveling wave solutions of the nonlocal profile equation \eqref{e:travode}.
The method of proof parallels that given for Theorem 1 in \cite{HLS} in the context of the fifth order Kawahara equation.

\begin{theorem}\label{t:existence}
Let $\alpha>\frac{1}{2}$, and $p\ge 1$ be fixed such that either $p\in\mathbb{N}$
or else $p=\frac{m}{n}$ where $m$ and $n$ are even and odd integers, respectively.  
Then there exists constants $a_0,b_0\in\RM$
such that for any fixed $b\in(-b_0,b_0)$ the nonlocal profile equation \eqref{e:travode}
admits a one-parameter family of even, periodic solutions $\{u_{a,b}\}_{a\in(-a_0,a_0)}$ of the form
\[
u_{a,b}(x)=P_{a,b}(k_{a,b}x)
\]
where $P_{a,b}$ is $2\pi$-periodic and smooth in its argument.  Moreover, the following properties hold:
\begin{itemize}
\item[(i)] The map $k:(-a_0,a_0)\times(-b_0,b_0)\to\RM$ is analytic, even in $a$, and satisfies
\[
k_{a,b}^\alpha=k^*(b)+\tilde{k}(a,b)
\]
where $k^*(b)=(p+1)Q_b^p-1$ and
\[
\tilde{k}(a,b)=\sum_{n\geq 1}\tilde{k}_{2n}(b)a^{2n},\quad\left|\tilde k_{2n}(b)\right|\leq\frac{K_0}{\rho_0^{2n}},
\]
for any $|(a,b)|\ll 1$ and some positive constants $K_0$ and $\rho_0>a_0$.
\item[(ii)] The map $(-a_0,a_0)\times(-b_0,b_0)\ni(a,b)\mapsto P_{a,b}\in H^\alpha(\RM/2\pi\ZM)$ is analytic
and can be expanded as
\[
P_{a,b}(z)=Q_b+a\cos(z)+\sum_{\substack{n,m\neq 0,~n+m\geq 2\\ n-m\neq\pm 1}}\tilde{p}_{n,m}(b)e^{i(n-m)z}a^{n+m}
\]
where $\tilde{p}_{n,m}\in\RM$ are such that $\tilde{p}_{n,m}(b)=\tilde{p}_{m,n}(b)$ with
\[
\left|\tilde p_{n,m}(b)\right|\leq\frac{C_0}{\rho_0^{n+m}}
\]
for any $|b|\ll 1$ and some $C_0>0$.
\item[(iii)] The Fourier coefficients $\hat p_n(a,b)$ of the $2\pi$-periodic function $P_{a,b}$,
\[
P_{a,b}(z)=\sum_{n\in\ZM}\hat p_n(a,b)e^{inz}
\]
are real and satisfy $\hat p_0(a,0)=1+\mathcal{O}(a^2)$ and $\hat{p}_n(a,0)=\mathcal{O}(|a|^n)$ for $n\neq 0$
as $a\to 0$.  Moreover, the map $a\mapsto\hat{p}_n(a,b)$ is even (resp. odd) for even (resp. odd)
values of $n$.
\end{itemize}
\end{theorem}

\begin{proof}
Renormalizing the period to $2\pi$, let $k\in\RM^+$ denote a wavenumber and set $z=kx$ so that
the rescaled profile equation becomes
\begin{equation}\label{e:profile2}
-k^\alpha\Lambda^\alpha u-u+u^{p+1}=b.
\end{equation}
We seek $2\pi$-periodic solutions of \eqref{e:profile2}.  To this end, first
notice that since $\alpha>\frac{1}{2}$ one can show by arguments identical that in [FL12, Appendix B]
that any $u\in H^\alpha(\RM/2\pi\ZM)$ satisfying \eqref{e:profile2} automatically lies in $H^{2\alpha+1}(\RM/2\pi\ZM)$.
Iterating now the identity 
\[
u=\left[k^\alpha(\Lambda^\alpha+1)\right]^{-1}\left(u^{p+1}+(k^\alpha-1)u-b\right)
\]
implies that $u\in H^\infty(\RM/2\pi\ZM)$, so that any $2\pi$-periodic $H^\alpha$ solution
of \eqref{e:profile2} is automatically a smooth function of $z$.  
Thus, it is sufficient to seek solutions of \eqref{e:profile2} which lie in $H^\alpha(\RM/2\pi\ZM)$.

Set $X^\alpha:=H^\alpha(\RM/2\pi\ZM)\times\RM^+\times\RM$, considered to be equipped with
the natural graph norm, and define the map $F:X^\alpha \to L^2(\RM/2\pi\ZM)$ by
\[
F(v,k,b):=-k^\alpha\Lambda^\alpha v-v+v^{p+1}-b,
\]
noting that $F$ is well-defined by Sobolev embedding.
Clearly then, the zeros of $F$ correspond to $2\pi$-periodic solutions of \eqref{e:profile2}, which
can be taken to be even functions of $z$ by applying an appropriate spatial translation.
First, we claim that $F$ is $C^1$ on $X^\alpha$.  Indeed, given
any $(v,k,b)\in X^\alpha$ we find that
\[
\frac{\partial F}{\partial v}=-k^\alpha\Lambda^\alpha-1+(p+1)v^p,\quad\frac{\partial F}{\partial k}=-\alpha k^{\alpha-1}\Lambda^\alpha
\]
and $\frac{\partial F}{\partial b}=-1$. 
Clearly $\frac{\partial F}{\partial k}$ and $\frac{\partial F}{\partial b}$ depend continuously on $(v,k,b)\in X^\alpha$,
and that $\frac{\partial F}{\partial v}$ depends continuously on $(k,b)\in\RM^+\times\RM$.
To see that $\frac{\partial F}{\partial v}$ depends continuously on $v\in H^\alpha(\RM/2\pi\ZM)$, notice that by 
Sobolev embedding we have for all $v_1,v_2\in H^\alpha(\RM/2\pi\ZM)$
\begin{align*}
\left\|v_1^{p+1}-v_2^{p+1}\right\|_{L^\infty(\RM/2\pi\ZM)}&\leq C\|v_1-v_2\|_{H^\alpha(\RM/2\pi\ZM)}\\
&\qquad\times    \left(\|v_1\|_{H^\alpha(\RM/2\pi\ZM)}^p+\|v_2\|_{H^\alpha(\RM/2\pi\ZM)}^p\right)
\end{align*}
for some constant $C>0$, from which we find
\begin{align*}
&\left\|\frac{\partial F}{\partial v}(v_1,k,b)f-\frac{\partial F}{\partial v}(v_2,k,b)f\right\|_{L^2(\RM/2\pi\ZM)}
\leq C\|v_1-v_2\|_{H^\alpha(\RM/2\pi\ZM)}\\
&\qquad\qquad\times    \left(\|v_1\|_{H^\alpha(\RM/2\pi\ZM)}^p+\|v_2\|_{H^\alpha(\RM/2\pi\ZM)}^p\right)\|f\|_{H^\alpha(\RM/2\pi\ZM)}
\end{align*}
for all $f\in H^\alpha(\RM/2\pi\ZM)$,
which establishes the continuity on $v$, as claimed.  Together, the above arguments verify
that $F$ is $C^1$ on $X^\alpha$.

Continuing, by inspection we see that $F(Q_b,k,b)=0$  for $|b|\ll 1$ and any $k\in\RM^+$.
Furthermore, for fixed $b\ll 1$ and $k>0$ we have
\[
\frac{\partial F}{\partial v}(Q_{b},k,b))=-k^\alpha\Lambda^\alpha-1+(p+1)Q_{b}^p,
\]
which, by Fourier analysis, has a trivial kernel in $L^2(\RM/2\pi\ZM)$ provided $k^\alpha\neq k^*(b)$,
where $k^*(b)$ is given in the statement of the theorem.
When $k^\alpha\neq k^*(b)$ then, the Implicit function theorem implies that the root $(Q_{b},k,b)$ of $F(v,k,b)$ continues
uniquely for $|k-k_0|\ll 1$ and $|b-b_0|\ll 1$.  By inspection, this must correspond to the nearby
equilibrium solutions $Q_b$.  
To find non-constant solutions then, we must consider the case $k^\alpha=k^*(b)$,
in which case the kernel of the above linear operator is two dimensional and is spanned by $e^{\pm iz}$.
We now construct periodic solutions $v$ to the equation $F(v,k,b)=0$ by using a Lyapunov-Schmidt reduction
for $|b|\ll 1$ and $|k^\alpha-k^*(b)|\ll 1$.

We set $k^\alpha=k^*(b)+\tilde k$ and
\begin{equation}\label{vexpand}
v(z)=Q_b+\frac{1}{2}Ae^{iz}+\frac{1}{2}\bar Ae^{-iz}+h(z)
\end{equation}
where $A\in\CM$ and $h\in H^\alpha(\RM/2\pi\ZM)$ satisfies
\[
\int_0^{2\pi}h(z)e^{\pm iz}dz=0.
\]
Substituting these expressions into the equation $F(v,k,b)=0$ leads to an equation of the form
\begin{equation}\label{perteqn}
L_bh=\mathcal{N}(h,A,\bar A,\tilde k,b),
\end{equation}
where
\[
L_{b}:=\frac{\partial F}{\partial v}(Q_{b},k^*(b),b)
\]
and $\mathcal{N}(0,0,0,\tilde k,b)=0$ for any $\tilde k,b\in\RM$.  
Again using Fourier analysis, we see that the kernel of $L_b$ is two 
dimensional and is spanned by $e^{\pm iz}$.  We denote by $P:L^2(\RM/2\pi\ZM)\to \ker(L_b)$ the spectral
projection onto the kernel of $L_b$, defined for any $u\in L^2(\RM/2\pi\ZM)$ by 
\[
Pu(z)=\hat u(1)e^{iz}+\hat u(-1)e^{-iz}.
\]
Since $Ph=0$ then, the perturbation equation \eqref{perteqn} is equivalent to the system
\begin{equation}\label{pertsystem1}
\left\{\begin{aligned}
L_bh&=\left({\rm I}-P\right)\mathcal{N}(h,A,\bar A,\tilde k,b)\\
0&=P\mathcal{N}(h,A,\bar A,\tilde k,b),
\end{aligned}\right.
\end{equation}
which is valid for any $|b|\ll 1$.

Notice the restriction of $L_b$ to $\mathcal{Z}:=({\rm Id}-P)H^\alpha(\RM/2\pi\ZM)$ has a bounded inverse given
by
\[
\left(L_b\big{|}_{\mathcal{Z}}^{-1}\right)v=\sum_{n\neq\pm 1}\frac{\hat v(n)}{k^*(b)^\alpha\left(1-|n|^\alpha\right)}
\]
for any $v\in\mathcal{Z}$.  In particular, this formula shows that the operators 
$\left(L_b\big{|}_{\mathcal{Z}}\right)^{-1}$ form a family of bounded linear operators
depending analytically on $b$.  Therefore, the system \eqref{pertsystem1} is equivalent to
\begin{equation}\label{pertsystem2}
\left\{\begin{aligned}
h&=L_b^{-1}\left({\rm Id}-P\right)\mathcal{N}(h,A,\bar A,\tilde k,b)\\
0&=P\mathcal{N}(h,A,\bar A,\tilde k,b).
\end{aligned}\right.
\end{equation}
Using the implicit function theorem \cite[Theorem 2.3]{CH}, we can solve the first equation in \eqref{pertsystem2} to find a unique
solution $h=H_*(A,\bar A,\tilde k,b)\in\mathcal{Z}$
that depends analytically on $(A,\bar A,\tilde k,b)$ in a neighborhood
of $(0,0,0,b_0)$ in the space ${\rm diag}(\CM^2)\times\RM^2$, where
${\rm diag}(\CM^2):=\{(z,\bar z):z\in\CM)$.  In particular, notice the uniqueness
of this solution implies that
\begin{equation}\label{e:horder}
H_*(0,0,\tilde k,b)=0
\end{equation}
for all $|b|\ll 1$.  Furthermore, the invariance of \eqref{e:travode}
with respect to spatial translations $z\mapsto z+z_0$ and the reflection $z\mapsto -z$ implies the relations
\begin{equation}\label{relations}
\left\{\begin{aligned}
H_*(A,\bar{A},\tilde k,b)(z+z_0)&=H_*(Ae^{iz_0},\bar Ae^{-iz_0},\tilde k, b)(z)\\
H_*(A,\bar{A},\tilde k,b)(-z)&=H_*(\bar A,A,\tilde k,b)(z).
\end{aligned}\right.
\end{equation}

Substituting the above into $P\mathcal{N}(h,A,\bar A,\tilde k,b)=0$ now yields the equation
\[
P\mathcal{N}(H_*(A,\bar{A},\tilde k,b)(z),A,\bar A,\tilde k,b)=0
\]
which must be solved.  Using the explicit form of the projection $P$, this equation has
solutions provided the two orthogonality conditions
\[
J_\pm(A,\bar A,\tilde k,b)=\int_0^{2\pi}\frac{Ae^{iz}\pm\bar Ae^{-iz}}{2}\mathcal{N}(H_*(A,\bar{A},\tilde k,b)(z),A,\bar A,\tilde k,b)dz=0\\
\]
are satisfied.
Notice that the relations \eqref{relations} imply that the functions $J_{\pm}$ satisfy
\begin{equation}\label{relations2}
\begin{aligned}
J_+(Ae^{iz_0},\bar Ae^{-iz_0},\tilde k,b)&=J_+(A,\bar A,\tilde k,b)=J_+(\bar A,A,\tilde k,b)\\
J_-(Ae^{iz_0},\bar Ae^{-iz_0},\tilde k,b)&=J_-(A,\bar A,\tilde k,b)=-J_-(\bar A,A,\tilde k,b).
\end{aligned}
\end{equation}
In particular, taking $z_0=-2\arg(A)$ in the equalities for $J_-$ implies that
\[
J_-(\bar A,A,\tilde k,b)=J_-(A,\bar A,\tilde k,b)=-J_-(\bar A,A,\tilde k,b),
\]
from which we see that the condition $J_-(A,\bar A,\tilde k,b)=0$ is always satisfied.

As for the solvability condition $J_+=0$, taking $z_0=-\arg(A)$ in \eqref{relations2}
we find $J_+(A,\bar A,\tilde k,b)=J_+(|A|,|A|,\tilde k,b)$
so that the associated solvability condition becomes $J_+(a,a,\tilde k,b)=0$ where $a\in\RM$ belongs
to a sufficiently small neighborhood of the origin.  Noting that \eqref{e:horder} implies that $a^{-1}H_*(a,a,\tilde k,b)$
is analytic in $a$ near $a=0$, it follows from the explicit form of the function $\mathcal{N}$ that
\begin{align*}
J_+(a,a,\tilde{k},b)&=\int_{0}^{2\pi}a\cos(z)\mathcal{N}(H_*(A,\bar{A},\tilde k,b)(z),A,\bar A,\tilde k,b)dz\\
&=a^2\left(\tilde{k}+\widetilde{J}(a,\tilde k,b)\right),
\end{align*}
where $\widetilde J$ is analytic in its argument, even with respect to the parameter $a$, and satisfies
$\widetilde J(0,0,b)=\partial_{\tilde k}\widetilde J(0,0,b)=0$.  By the Implicit Function Theorem \cite{CH} again, for $0<|a|\ll 1$ we obtain
a solution $\tilde{k}(a,b)$ of $\widetilde{J}(a,\tilde k,b)=-\tilde k$, and hence of 
 $J_+(a,a,\tilde k,b)=0$, defined for sufficiently small $a,b\in\RM$.
Furthermore, it follows that $\tilde{k}$ is even in $a$ and analytic in a sufficiently small
neighborhood of the origin in $\RM^2$, so that the function
$k(a,b)=\left(k^*(b)+\tilde{k}(a,b)\right)^{1/\alpha}$ satisfies the properties discussed in (i).

From the above considerations, it follows that the system \eqref{pertsystem2} has a unique solution
\[
(h,\tilde{k})=\left(H_*(A,\bar{A},\tilde k(|A|,b),b),\tilde{k}(|A|,b)\right),
\]
defined for any sufficiently small $A\in\CM$ and $|b|\ll 1$.  Substituting $h=H_*(A,\bar{A},\tilde k(|A|,b),b)$
into \eqref{vexpand} thus yields a $2\pi$-periodic solution of \eqref{e:travode}.  The periodic
solutions $P_{a,b}$ described in the theorem are now found by restricting to $A\in\RM$, so that
\[
P_{a,b}(z)=Q_b+a\cos(z)+v_{a,b}(z),\qquad v_{a,b}(z)=H_*(a,a,\tilde k(a,b),b)(z).
\]
The properties of $P_{a,b}$ described in (ii) are now easily deduced from analyticity and the symmetries
of the function $H_*(A,\bar{A},\tilde k(|A|,b),b)$, while (iii) follows directly from the expansion of
$P_{a,b}$ given in (ii).
\end{proof}

\section{Proof of Proposition \ref{bloch}}\label{a:bloch}

In this appendix, we establish a nonlocal type of Floquet-Bloch theory that is suitable for the stability
analysis presented in this paper.

\begin{proof}[Proof of Proposition \ref{bloch}]
Clearly (iii) holds for some $\xi\in[-1/2,1/2)$ if and only if the kernel of the operator $\mathcal{M}_{a,b,\xi}-\lambda{\bf I}$
acting on $L^2(\RM/2\pi\ZM)$ is non-trivial.  Furthermore, the operator $\mathcal{M}_{a,b,\xi}$ acting
on $L^2(\RM/2\pi\ZM)$ has densely defined and compactly embedded domain $H^{\alpha+1}(\RM/2\pi\ZM)$ and hence
has compact resolvent.  As a result, the spectrum of $\mathcal{M}_{a,b,\xi}$ consists of isolated eigenvalues
of finite multiplicity and, in particular, we see that $\lambda\in\sigma\left(\mathcal{M}_{a,b,\xi}\right)$ 
if and only if the operator $\mathcal{M}_{a,b,\xi}-\lambda{\bf I}$ has a nontrivial kernel.  This establishes
that (ii)$\iff$(iii).

Now, assume that (ii) does not hold, i.e. that 
the operator $\mathcal{M}_{a,b,\xi}-\lambda{\bf I}$ is boundedly invertible on $L^2(\RM/2\pi\ZM)$
for all $[-1/2,1/2)$.
Then there exists a constant $C>0$ such that
\[
\left\|\left(\mathcal{M}_{a,b,\xi}-\lambda{\bf I}\right)v\right\|_{L^2(\RM/2\pi\ZM)}\geq C\|v\|_{L^2(\RM/2\pi\ZM)}
\]
for all $v\in H^{\alpha+1}(\RM/2\pi\ZM)$ and $\xi\in[-1/2,1/2)$.  Using \eqref{par} then, we find 
for all $w\in H^{\alpha+1}(\RM)$ that
\begin{align*}
\left\|\left(\mathcal{M}_{a,b}-\lambda{\bf I}\right)w\right\|^2_{L^2(\RM)}
   &=2\pi\int_{-1/2}^{1/2}\left\|\left(\mathcal{M}_{a,b,\xi}-\lambda{\bf I}\right)\check{w}(\xi,\cdot)\right\|_{L^2(\RM/2\pi\ZM)}^2d\xi\\
&\geq 2\pi C^2\int_{-1/2}^{1/2}\left\|\check{w}(\xi,\cdot)\right\|_{L^2(\RM/2\pi\ZM)}^2d\xi\\
&=C^2\|w\|_{L^2(\RM)}.
\end{align*}
It follows then that $\mathcal{M}_{a,b}-\lambda{\bf I}$ is boundedly invertible
as an operator acting on $L^2(\RM)$.  This establishes that (i)$\Rightarrow$(ii).

Finally, assume that (ii) holds, i.e. that for some $\xi_0\in[-1/2,1/2)$ the operator
$\mathcal{M}_{a,b,\xi_0}-\lambda{\bf I}$ is not boundedly invertible on $L^2(\RM/2\pi\ZM)$.
Then by the above considerations there exists a $v\in H^{\alpha+1}(\RM/2\pi\ZM)$ such that 
$\left(\mathcal{M}_{a,b,\xi_0}-\lambda{\bf I}\right)v=0$.  
For each $0<\eps<1$, let $\phi_\eps:\RM\to\RM$ be defined by
\[
\phi_\eps(\xi)=\left\{\begin{array}{ll}
               \eps^{-1/2}&\textrm{ if }|\xi|<\eps/2\\
               0& \textrm{ if } |\xi|\geq\eps/2,
               \end{array}\right.
\]
and note that $\|\phi_\eps\|_{L^2(\RM)}=1$.
Given a fixed $\eps\in(0,1)$ then, notice that the function
\[
\RM^2\ni(\xi,z)\mapsto v(z)\phi_\eps(\xi-\xi_0)\in\CM
\]
belongs to $L^2([-1/2,1/2);L^2(\RM/2\pi\ZM))$, and can hence be viewed as the Bloch-transform
of some function $v_{\eps}\in L^2(\RM)$ with
\[
\left\|v_{\eps}\right\|_{L^2(\RM)}^2=\int_{-1/2}^{1/2}\phi_\eps(\xi-\xi_0)^2\|v\|_{L^2(\RM/2\pi\ZM)}^2d\xi=\|v\|_{L^2(\RM/2\pi\ZM)}^2
\]
for all $\eps>0$ sufficiently small.
Then for each $0<\eps\ll 1$, we find using \eqref{par} that
\begin{align*}
\left\|\left(\mathcal{M}_{a,b}-\lambda{\bf I}\right)v_{\eps}\right\|_{L^2(\RM)}^2
&=\int_{-1/2}^{1/2}\phi_\eps(\xi-\xi_0)^2\left\|\left(\mathcal{M}_{a,b,\xi}-\lambda{\bf I}\right)v(\cdot)\right\|_{L^2(\RM/2\pi\ZM)}^2d\xi\\
&=\frac{1}{\eps}\int_{|\xi-\xi_0|<\eps/2}\left\|\left(\mathcal{M}_{a,b,\xi}-\lambda{\bf I}\right)v(\cdot)\right\|_{L^2(\RM/2\pi\ZM)}^2d\xi
\end{align*}
Next, we want to show the above quantity tends to zero as $\eps\to 0^+$.

To this end, notice that Lemma \ref{l:blochexpand} in Section \ref{s:crit} implies 
that the mapping $\xi\to\mathcal{M}_{a,b,\xi}$ is continuous in the operator
norm from $H^{\alpha+1}(\RM/2\pi\ZM)$ to $L^2(\RM/2\pi\ZM)$.  Indeed, for a given $w\in H^{\alpha+1}(\RM/2\pi\ZM)$
and $\xi_1,\xi_2\in[-1/2,1/2)$ we have the estimate
\[
\left\|\left(\mathcal{M}_{a,b,\xi_1}-\mathcal{M}_{a,b,\xi_2}\right)w\right\|_{L^2(\RM/2\pi\ZM)}
\lesssim |\xi_1-\xi_2|\|v\|_{H^{\alpha+1}(\RM/2\pi\ZM)},
\]
yielding the desired continuity.  It now follows that
\begin{align*}
&\lim_{\eps\to 0^+}
\frac{1}{\eps}\int_{|\xi-\xi_0|<\eps/2}\left\|\left(\mathcal{M}_{a,b,\xi}-\lambda{\bf I}\right)v(\cdot)\right\|_{L^2(\RM/2\pi\ZM)}^2d\xi\\
&\qquad\qquad=\left\|\left(\mathcal{M}_{a,b,\xi_0}-\lambda{\bf I}\right)v\right\|_{L^2(\RM/2\pi\ZM)}^2=0,
\end{align*}
so that, in particular, given any $n\gg 1$ there exists an $\eps_n>0$ such that
\[
\left\|\left(\mathcal{M}_{a,b}-\lambda{\bf I}\right)v_{\eps_n}\right\|_{L^2(\RM)}^2<\frac{1}{n}.
\]
Recalling that $\|v_{\eps}\|_{L^2(\RM)}=\|v\|_{L^2(\RM/2\pi\ZM)}$ for all $0<\eps\ll 1$, it follows
that the operator $\mathcal{M}_{a,b}-\lambda{\bf I}$ is not boundedly invertible.
This establishes (ii)$\Rightarrow$(i), which completes the proof.
\end{proof}

\end{document}